\documentclass[12pt,a4paper]{article} 
 \usepackage{hyperref}

\usepackage{a4wide,amsmath,amssymb,amsthm,mathtools,
xcolor}
\numberwithin{equation}{section}
\usepackage[all]{xy}
\usepackage{tikz}
\usetikzlibrary{cd}
\usepackage[mathcal]{euscript}
\usepackage{mathrsfs}
\usepackage{enumitem}
\usepackage[capitalise]{cleveref} 
\let\mathcal=\mathscr

\definecolor{aloe}{rgb}{0.0, 0.5, 0.0}
\newtheorem{thm}{Theorem}[section]

\newtheorem{lemma}[thm]{Lemma}

\newtheorem{prop}[thm]{Proposition}

\theoremstyle{remark}
\newtheorem{rem}[thm]{Remark}
\newtheorem{ex}[thm]{Example}
\newenvironment{remark}{\begin{rem}\rm}{\parbox{2mm}{\hfill}\qee\end{rem}}

\newcommand{\parag}[1]{\paragraph{#1.}}
\newcommand{\E}{\mathfrak E}

\newcommand{\grass}[2]{\operatorname{Gr}_{#1}(#2)}
\renewcommand{\Im}{\operatorname{Im}}

\newcommand{\rk}{\operatorname{rk}}

\newcommand{\C}{{\mathbb C}}

\newcommand{\Q}{{\mathbb Q}}

\newcommand{\qee}{\mbox{\hspace{0.2mm}}\hfill$\triangle$}

\newcommand{\tr}{\operatorname{tr}}

\newcommand{\Oc}{\mathcal O}          % structure sheaf
\newcommand{\Gr}{\mathrm{Gr}}         % Grassmann bundle of subbundles
\newcommand{\HGr}{\mathfrak{Gr}}      % Higgs Grassmannian
\newcommand{\Ps}{\mathbb P}
\begin{document}
\begin{center} \bf  RANK 3 CURVE SEMISTABLE HIGGS BUNDLES: 
\\[3pt] THE NILPOTENT CASE\end{center}
 
\thispagestyle{empty} \vspace{-3mm}
\begin{center}{\sc Ugo Bruzzo$^{ab}$  and Beatriz Gra\~na Otero$^{c}$}  \\[5pt]
\small
$^a$ Departamento de Matem\'atica, Instituto de Ci\^encias \\
Exatas,  Universidade Federal    de Minas Gerais, \\
Av.~Ant\^onio Carlos 6627, Belo Horizonte, 31270-901 MG, Brazil \\
$^b$ IGAP (Institute for Geometry and Physics), Trieste, Italy  \\
$^c$ Departamento de Matem\'aticas  and IUFFYM (Instituto de \\
F\'\i sica Fundamental  y Matem\'aticas),  Universidad de Salamanca, \\
Plaza de la Merced 1-4, 37008 Salamanca, Spain 
\\[3pt]
Email: {\tt ubruzzo@ufmg.br,   beagra@usal.es}

%}
\end{center}

\vfill

\begin{abstract}
A Higgs bundle on a smooth complex projective variety is curve semistable if its
pullback to every smooth projective curve is semistable. A 2008 conjecture
asserts that this is equivalent to slope semistability for some polarization plus
vanishing discriminant. One direction is known, the converse holds in rank two,
and Huang recently gave a rank-four counterexample. We
prove the conjecture for rank three nilpotent Higgs bundles in any dimension.
We treat separately the two possible cases: if the Higgs field $\phi$ is nilpotent of order two,
 there is a divisorial component of a Higgs Grassmannian where nefness of the
normalized tautological class forces the discriminant to vanish; otherwise, $\phi$ is regular nilpotent, and then the canonical kernel flag,
commutativity and Bogomolov's lemma on line subbundles of $\Omega^1_X$
complete the proof.
\end{abstract}

\vfill 

%\tableofcontents

\vfill

\noindent\begin{minipage}{\textwidth} \small
\parbox{\textwidth}{\hrulefill} \\
Date: \today 
Research partly supported by Bolsa de Produtividade 305343/2025-4 of Brazilian CNPq (UB),   Grant PID2021-128665NB-I00 funded by MCIN/AEI/10.13039/501100011033 (Spain),   ``ERDF A way of making Europe,''  Universidad de Salamanca through Programa XIII (BGO).

\end{minipage}

\newpage

\section{Introduction}

\parag{The conjecture}
Let $X$ be a smooth complex projective variety and let $\mathfrak E=(E,\phi)$ be a Higgs bundle on $X$ of rank $r$, that is, a vector bundle $E$ with a morphism $\phi\colon E\to E\otimes\Omega^1_X$ satisfying the commutativity condition $\phi\wedge\phi=0$, where 
$ \phi  \wedge\phi \colon E\to E\otimes\Omega^2_X$ 
is the morphism induced by $\phi$. 
Write
\[
\Delta(E)=c_2(E)-\frac{r-1}{2r}\,c_1(E)^2\in H^4(X,\mathbb Q)
\]
for the discriminant of $E$. For a morphism $f\colon C\to X$ from a smooth connected projective curve, the pullback $f^*\mathfrak E$ is the Higgs bundle $(f^*E,f^*\phi)$ on $C$ whose Higgs field is the pullback of $\phi$ contracted along $df$, and $\mathfrak E$ is said to be \emph{curve semistable} if $f^*\mathfrak E$ is a semistable Higgs bundle on $C$ for every such $f$. The two conditions
\begin{enumerate}\itemsep=-2pt
\item[(1)] $\mathfrak E$ is semistable with respect to some polarization of $X$ and $\Delta(E)=0$;
\item[(2)] $\mathfrak E$ is curve semistable
\end{enumerate}
are the subject of this paper. For ordinary vector bundles, that is for $\phi=0$, they are equivalent: this is the theorem of Nakayama and of the first author and Hern\'andez Ruip\'erez \cite{Nakayama99,BHR06}. For Higgs bundles the implication (1) $\Rightarrow$ (2) was proved in \cite{BHR06}. The converse was conjectured in \cite{BGO11}; since curve semistability implies semistability with respect to every polarization, its content is the statement that
\begin{equation}\label{eq:conj}
\text{\em a curve semistable Higgs bundle has vanishing discriminant.}
\end{equation}

Before the present work the conjecture was known in the following cases: for Higgs bundles on the ``Higgs varieties'' of \cite{BLG16,LLG17}, a class of varieties defined by the validity of the conjecture on them, which includes for instance  varieties with numerically effective tangent bundle, and simply connected Calabi-Yau varieties \cite{BL19,BC23}; for Higgs bundles of rank two on an arbitrary $X$ \cite{BGH23}; and, more generally, on an arbitrary $X$ and in arbitrary rank, for Higgs bundles possessing a Higgs Grassmannian with an irreducible component which is a divisor in the ambient Grassmann bundle and dominates $X$ \cite[Thm.~3.3]{BGH23}. 
 
 Curve semistability is equivalent to the nefness of the normalized tautological classes $\theta_s(\mathfrak E)$ on the schemes $\mathfrak{Gr}_{r-s}(\mathfrak E)$ of rank $s$ Higgs quotients \cite{BHR06,BGO11}; on a dominating divisorial component when $X$ is a surface, of such a scheme the top self-intersection of $\theta_s(\mathfrak E)$ is a negative multiple of $\Delta(E)$, so that nefness gives $\Delta(E)\le0$, while Bogomolov's inequality for semistable Higgs bundles \cite{Simpson88,Simpson92} gives $\Delta(E)\ge0$, so that $\Delta(E)=0$. Then this is extended to arbitrary dimension of $X$.
 
\parag{Huang's counterexample}
The conjecture is false in general. Huang \cite{Huang26} has constructed, on the second symmetric product $X=\Sigma^{(2)}$ of a very general plane quintic $\Sigma$, a rank four Higgs bundle $\mathfrak E_s=(E_s,\phi_s)$ with nilpotent Higgs field, $\phi_s\circ\phi_s=0$, which is curve semistable, has trivial determinant, and $c_2(E_s)=\Delta(E_s)=10$. The underlying bundle is a nonsplit extension $0\to F^*\to E_s\to F\to0$ of the tautological rank two bundle $F$ on $X$ by its dual, and the Higgs field is the composite of the projection $E_s\to F$, a natural morphism $\psi\colon F\to F^*\otimes\Omega^1_X$, and the inclusion $F^*\to E_s$. The extension $E_s$ is determined by a section $s$ of $\Oc_\Sigma(1)$, which is chosen so as to make $E_s$ curve semistable.   The kernel $\phi_s$ and the image of the associated morphism $E_s\otimes T_X\to E_s$ coincide, both being $F^*$, and the Higgs Grassmannians of $\mathfrak E_s$ have no dominating divisorial component,

\parag{The main result}
In this paper we prove the conjecture for Higgs bundles of rank three whose Higgs field is nilpotent. Here $\phi$ is called nilpotent if for every local vector field $t$ the endomorphism $\phi_t=\langle\phi,t\rangle$ of $E$ is nilpotent, or equivalently, if the coefficients of the characteristic polynomial of $\phi$, which are sections of the symmetric powers $S^i\Omega^1_X$, all vanish.

The main content of this paper is the following result.
\begin{thm}\label{thm:main}
Let $\mathfrak E=(E,\phi)$ be a Higgs bundle of rank three with nilpotent Higgs field on a smooth complex projective variety $X$ of any dimension. Then $\mathfrak E$ is curve semistable if and only if it is semistable with respect to some polarization and $\Delta(E)=0$.
\end{thm}

Since the Higgs field of a rank three Higgs bundle is nilpotent as soon as its trace-free part has vanishing $\sigma_2\in H^0(S^2\Omega^1_X)$ and $\sigma_3\in H^0(S^3\Omega^1_X)$, and since subtracting $\frac13\operatorname{tr}\phi\otimes\mathrm{id}$ changes neither semistability nor the discriminant, Theorem~\ref{thm:main} gives the conjecture in rank three outright on every variety with $H^0(X,S^2\Omega^1_X)=H^0(X,S^3\Omega^1_X)=0$. This is the case, for instance, for smooth surfaces in $\mathbb P^3$,
which is of some interest because in degree at least five such surfaces are of general type, a class of varieties on which the conjecture had so far been verified in no nontrivial instance for rank higher than two.

\parag{Sketch of the proof} As a first step, we study the case of a rank three system of Hodge bundles, i.e., we assume that $E = E_1 \oplus E_2$, where $E_1$ and $E_2$ have ranks one and two (in some order), and the Higgs field $\phi$ maps $E_1$ to $E_2 \otimes \Omega^1_X$ and vanishes on $E_2$.
When $\mathfrak E=(E, \phi)$ is a rank three nilpotent Higgs bundle there are two possibilities: either $\phi \circ \phi=0$ or some contraction $\phi_t$ with the local vector field $t$ is a regular nilpotent endomorphism at the general point of $X$. The first situation reduces to the previous case of a system of Hodge bundles, while the second is treated in a different way, using a result by Bogomolov about the line subbundles of the cotangent bundle.

Both for systems of Hodge bundles and for general nilpotent Higgs bundles of rank three, a central role is played by the analysis of the structure of the Higgs Grassmannians (see Section \ref{sec:hgr}).

\parag{Conventions}
All varieties are smooth and projective over $\mathbb C$ unless stated otherwise. For a Higgs bundle $\mathfrak E=(E,\phi)$ and a local vector field $t$ we write $\phi_t=\langle\phi,t\rangle\in\operatorname{End}(E)$. Semistability on $X$ is slope semistability with respect to a polarization. The Grassmann bundle of rank $d$ subbundles of $E$ is $\mathrm{Gr}_d(E)$;  $\mathbb PE=\mathrm{Gr}_1(E)$ parameterizes lines in $E$ and $\mathbb P(E^\vee)=\mathrm{Gr}_{r-1}(E)$ rank one quotients of $E$. $\mathfrak{Gr}_d(\mathfrak E)\subset\mathrm{Gr}_d(E)$ is the Higgs Grassmannian of $\phi$-invariant rank $d$ subbundles, and it is also the scheme of rank $r-d$ Higgs quotients. Systems of Hodge bundles are written $E=\bigoplus_pE_p$ with $\phi(E_p)\subset E_{p-1}\otimes\Omega^1_X$.

\parag{AI declaration} Claude Fable 5.1 was used to assist with the initial drafting of specific sections and to expedite computational checks, while the final logical structure and mathematical validation remain entirely the authors' own.

\section{Higgs Grassmannians}\label{sec:hgr}

We collect here what will be needed about Higgs Grassmannians and their tautological classes, following \cite{BGH23,BGG26}.

\parag{Definition and local equations}
Let $\mathfrak E=(E,\phi)$ be a Higgs bundle of rank $r$ on a smooth variety $X$ of dimension $n$, and for $0<d<r$ let $\pi_d\colon\Gr_d(E)\to X$ be the Grassmann bundle of rank $d$ subbundles of $E$, with universal subbundle $S_d\subset\pi^*E$ and universal quotient $Q_d=\pi_d^*E/S_d$. The pullback $\pi_d^*\phi$ induces a morphism $S_d\to\pi_d^*E\otimes\pi_d^*\Omega^1_X\to Q_d\otimes\pi^*\Omega^1_X$, i.e.\ a section of $\mathcal Hom(S_d,Q_d)\otimes\pi^*\Omega^1_X$, and the \emph{Higgs Grassmannian} $\HGr_d(\mathfrak E)\subset\Gr_d(E)$ is its zero scheme. Its points over $x\in X$ are the $d$-dimensional subspaces $W\subset E_x$ with $\phi_x(W)\subset W\otimes\Omega^1_{X,x}$, i.e.\ invariant under all the contractions $\phi_{t}$, $t\in T_xX$; it is the natural parameter space of the rank $d$ Higgs subbundles of $\mathfrak E$, and equally of the rank $s=r-d$ Higgs quotients, a quotient $E\to M$ being a Higgs quotient exactly when its kernel is $\phi$-invariant. In a local frame $e_1,\dots,e_r$ of $E$, with $\phi(e_j)=\sum_i\phi_{ij}\otimes e_i$ and $\phi_{ij}$ local $1$-forms, and with vertical homogeneous coordinates $(z_1,\dots,z_r)$ on $\Ps E=\Gr_1(E)$, the Higgs Grassmannian of lines is cut out by the $1$-form valued equations
\begin{equation}\label{eq:hgrsub}
\sum_{k=1}^rz_k\bigl(\phi_{jk}z_i-\phi_{ik}z_j\bigr)=0,\qquad 1\le i<j\le r,
\end{equation}
which express $v\wedge\phi(v)=0$ for $v=\sum z_ke_k$ \cite[eq.\ (4)]{BGG26}; in local coordinates on $X$, each equation   splits into $n$ scalar ones. For rank one quotients, parameterized by $\Ps(E^\vee)=\Gr_{r-1}(E)$, the equations are obtained by transposing the Higgs field,
% \begin{equation}\label{eq:hgrquot}
% \sum_{k=1}^rw_k\bigl(\phi_{kj}w_i-\phi_{ki}w_j\bigr)=0,
% \end{equation}
and the identification $\HGr_{r-1}(\mathfrak E)=\HGr_1(\mathfrak E^\vee)$ with the Higgs Grassmannian of lines of the dual Higgs bundle $\mathfrak E^\vee=(E^\vee,\phi^t)$ is scheme-theoretic. For intermediate $d$ one uses the Pl\"ucker embedding, as in \cite[eq.\ (5)]{BGG26}.

The scheme structure may not  be reduced, and the Higgs Grassmannian is in general far from being equidimensional. Both phenomena are visible in the simplest examples \cite{BGG26}.

\parag{Tautological classes and curve semistability}
Let $\pi_s\colon\Gr_{r-s}(E)\to X$, for $0<s<r$,  be the Grassmann bundle of rank $s$ quotients, $Q_s$ the universal rank $s$ quotient, and
\begin{equation}\label{eq:theta}
\theta_s(E)=c_1(Q_s)-\frac sr\,\pi_s^*c_1(E)\ \in N^1(\Gr_{r-s}(E))_{\mathbb Q},
\end{equation}
the \emph{normalized tautological class}. Its restriction to the scheme of rank $s$ Higgs quotients is the class denoted $\theta_s(\mathfrak E)$ in \cite{BHR06}. For an ordinary vector bundle the nefness of $\theta_s(E)$ on the whole Grassmann bundle, for one and hence for all $s$, is equivalent to curve semistability \cite{BHR06};  this is a higher dimensional generalization of Miyaoka's criterion for semistability of vector bundles on curves \cite{Mi}.

For Higgs bundles the same mechanism, with Higgs quotients in place of quotients, gives \cite{BHR06,BGO11}:
\begin{center}
\emph{$\mathfrak E$ is curve semistable if and only if $\theta_s(\mathfrak E)$ is nef on $\HGr_{r-s}(\mathfrak E)$ for all $0<s<r$.}
\end{center}
Here the quantifier over all $s$ cannot be dropped: unlike the case of vector bundles, the class $\theta_s(\mathfrak E)$ lives on the possibly very small subscheme $\HGr_{r-s}(\mathfrak E)$ and controls only the Higgs quotients of rank $s$. For instance one can construct rank 3 examples where $\theta_1(\mathfrak E)$
is nef but the Higgs bundle is destabilized by a rank 2 Higgs quotient \cite{BHR06}.

\parag{The theorem of \cite{BGH23}}
The relation between the normalized tautological classes and the discriminant is the following. On a surface, writing $h$ for the relative hyperplane class of $\pi\colon \Ps(E^\vee)\to X$ and $\theta=h-\frac1r\pi^*c_1(E)$, one computes   $\pi_*\theta^{\,r+1}=-\Delta(E)$.  Thus for a curve semistable vector bundle,
$\Delta(E)\le 0$; but curve semistable bundles are semistable, 
so $\Delta(E)\ge 0$ by the Bogomolov inequality, and then 
$\Delta(E) =  0$.\footnote{On a surface we think of $\Delta(E)$ as a number.} This provides a quick proof of the implication
(2) $\Rightarrow$ (1) for   vector bundles (this is part of the theorem about curve semistable vector bundles in 
\cite{Nakayama99,BHR06}). Thus curve semistability for vector bundles is equivalent to the class $\theta$ being nef but not big.

For a Higgs bundle nefness is only available on $\HGr_{r-s}(\mathfrak E)$, and the question becomes what a component of the Higgs Grassmannian says about the top self-intersections of the classes $\theta_s$ in $\Gr_{r-s}(E)$. Let $W\subset\HGr_{r-s}(\mathfrak E)$ be an irreducible component, with its reduced structure, of codimension $1$ in $\Gr_{r-s}(E)$ over a surface, and write its class in $A^1(\Gr_{r-s}(E),\Q)$ by the
Leray--Hirsch theorem for Chow groups \cite[Prop.~14.6.5]{Fulton} as 
\begin{equation}\label{eq:Wclass}
[W]=\beta_0\,\theta_s+\pi_s^*\beta_1 
\end{equation}
with $\beta_i\in A^i(X,\mathbb Q)$, with $i=0,1$; here $\pi_s$ is the projection $\Gr_{r-s}(E)\to X$.
Assume that
$W$ surjects onto $X$.
Then $\beta_0>0$, and 
\begin{equation}\label{eq:thetaW}
\int_W\theta_s(\mathfrak E)^{\dim W}=-c\,\beta_0 \int_X  \Delta(E), \qquad c > 0.
\end{equation}
Reasoning as before, 
one gets $\Delta(E)=0$. For the reader's convenience we give a full statement of this result, and a sketch of the proof.

\begin{prop}\emph{\cite[Thms.~3.3 and 4.9]{BGH23}} \label{prop:BGH}
Let $\mathfrak E$ be a Higgs bundle of rank $r>1$ on a smooth projective variety $X$ of dimension $n\ge 2$, and suppose that for some $s$, $0<s<r$,  the Higgs Grassmannian $\HGr_{r-s}(\mathfrak E)$ has an irreducible component $\Sigma$ which is a divisor in $\Gr_{r-s}(E)$ and dominates $X$. Then the implication (2) $\Rightarrow$ (1) of the Conjecture holds.
\end{prop}

\begin{proof} 
Curve semistability implies semistability with respect to every polarization and the nefness of $\theta_s(\mathfrak  E)$, hence of its restriction to $\Sigma$. Let $H$ be an ample class on $X$.

\emph{Step 1: $n=2$.} 
It was proved in \cite{BGH23} that in this case $\beta_0=1$. Moreover, 
set $m=s(r-s)+1=\dim\Sigma \ge 2$.
% and let $f^\varepsilon$ be the number of standard Young tableaux of shape $\varepsilon=((r-s)^s)$, i.e.\ $\varepsilon=(2)$ for $s=1$ and $\varepsilon=(1,1)$ for $s=2$; in both cases $f^\varepsilon=1$. 
The computation in  \cite[Thm.~3.3]{BGH23} yields.
\begin{equation}\label{eq:thetacube}
\int_\Sigma\theta_s(\mathfrak  E)^m=-\frac{m(m+1)(m-1)}{r(r+1)(r-1)}\,f^{\varepsilon}\,\Delta(E)=-c\,\Delta(E)
\end{equation}
where $f^{\varepsilon}$ and $c$ are   positive coefficients, the former related to some Young tableaux combinatorics.
By the nefness of the class $\theta_s(\mathfrak  E)$ one has
$\Delta(E) \le 0$, while $\Delta(E)\ge 0$ since $\mathfrak E$ is semistable hence satisfies the Bogomolov inequality, so that $\Delta(E)=0$.

\emph{Step 2: $\Delta(E)\cdot H^{n-2}=0$ for $n>2$.} Let $Y\subset X$ be the intersection of $n-2$ generic divisors in $|mH|$, $m\gg0$, a smooth connected projective surface. The restriction $\mathfrak  E|_Y$ is a curve semistable Higgs bundle, and $\Sigma|_Y$ is a divisorial component of a Higgs Grassmannian of $\mathfrak  E|_Y$. By Step 1, $\Delta(E|_Y)=0$, i.e., $\Delta(E)\cdot H^{n-2}=0$.

\emph{Step 3.} Consider the Higgs bundle $\mathrm{End}(\mathfrak  E)=\mathfrak  E\otimes\mathfrak  E^\vee$. It is $H$-semistable because $\mathfrak  E$ is \cite{Simpson92}, it has $c_1(\mathrm{End}\,E)=0$, and
$$
\mathrm{ch}_2(\mathrm{End}\,E)=2r\,\mathrm{ch}_2(E)-c_1(E)^2=-2r\,\Delta(E),
$$
so that $\mathrm{ch}_2(\mathrm{End}\,E)\cdot H^{n-2}=0$. By Simpson's theorem \cite[Thm.~2]{Simpson92}, $\mathrm{End}(\mathfrak  E)$ admits a filtration whose quotients  are stable Higgs bundles with vanishing Chern classes; hence all Chern classes of $\mathrm{End}(E)$ vanish, in particular $\mathrm{ch}_2(\mathrm{End}\,E)=0$, i.e., $\Delta(E)=0$ in $H^4(X,\mathbb Q)$.  
\end{proof}

So  the nef class $\theta_s(\mathfrak  E)$ has zero top self-intersection  on $\Sigma$, as $\Delta(E)=0$: it is nef but not big --- as it happens on the whole Grassmann bundle $\grass{r-s}{E}$ in the vector bundle case.

\section{Systems of type (2,1) and (1,2) on a surface}
According to Simpson \cite{Simpson92} a \emph{system of Hodge bundles} is a direct sum of vector bundles $E=\bigoplus_{r,s} E_{r,s}$ with a Higgs field $\phi\colon E_{r,s} \to E_{r-1,s+1}\otimes\Omega_X^1$. As a preliminary for proving our result on nilpotent rank three Higgs bundles, we consider systems of Hodge bundles of a certain restricted type.

In this section $X$ is a smooth surface ($n=2$), $r=\rk E =3$, and $E=L\oplus F$ with $\rk L=1$, $\rk F=2$. Concerning the action of the Higgs field $\phi$ there are two cases:
\begin{itemize} \itemsep=4pt
\item[\textbf{(2,1)}] $\phi \colon F\to L\otimes\Omega^1_X$, $\phi |_L=0$; 
\item[\textbf{(1,2)}] $\phi \colon L\to F\otimes\Omega^1_X$, $\phi |_F=0$.
\end{itemize}
We choose a frame $e_1\in L$, $e_2,e_3\in F$ and in local coordinates $(x_1, x_2)$  on $X$ write the two $1$-forms
$$
a=a_1dx_1+a_2dx_2,\qquad b=b_1dx_1+b_2dx_2
$$
that determine $\phi $: in case (2,1) $\phi (e_2)=a\otimes e_1$, $\phi (e_3)=b\otimes e_1$, and in case (1,2) $\phi (e_1)=a\otimes e_2+b\otimes e_3$. Thus
\begin{equation}\label{eq:matrices}
\phi _{(2,1)}=\begin{pmatrix}0&a&b\\0&0&0\\0&0&0\end{pmatrix},\qquad
\phi _{(1,2)}=\begin{pmatrix}0&0&0\\a&0&0\\b&0&0\end{pmatrix}=\phi _{(2,1)}^{t}.
\end{equation} 
In both cases the commutativity condition $\phi \wedge\phi =0$ is automatic. In case (2,1), $(L,0)$ is a Higgs subbundle and $(F,0)$ a Higgs quotient; in case (1,2) it is the other way round.

A basic invariant is the $2$-form
\begin{equation}\label{eq:delta}
a\wedge b=(a_1b_2-a_2b_1)\,dx_1\wedge dx_2 ,
\end{equation}
which is the local expression of the determinant $\delta$ of $\phi $ regarded as a morphism between rank two bundles: $\delta\in H^0(X,L^{2}\otimes K_X\otimes\det F^{-1})$ in case (2,1) and $\delta\in H^0(X,\det F\otimes K_X\otimes L^{-2})$ in case (1,2) --- here $K_X$ is the canonical bundle of $X$.\footnote{The morphisms are $F\to L\otimes\Omega^1_X$ in the (2,1) case, and $L\otimes T_X\to F$ in the (1,2) case.}
We write $D=\{\delta=0\}$ for the degeneracy divisor and $Z=\{\phi =0\}\subset D$ for the vanishing locus of $\phi $. %Note that $a\wedge b(x)=0$ if and only if the matrices $\phi ^{(1)}(x),\phi ^{(2)}(x)$ are proportional, i.e.\ $\phi (x)=\omega\otimes N$ for a $1$-form $\omega$ and a rank $\le1$ endomorphism $N$.

\parag{Case (2,1): rank one Higgs quotients}\label{sec:A1}

\begin{prop}\label{prop:A1} 
Let $\mathfrak  E$ be of type   \textup{(2,1)}. In the vertical homogeneous coordinates $(w_1,w_2,w_3)$ on $\Ps(E^\vee)$ the ideal of the Higgs Grassmannian $\HGr_2(\mathfrak  E)$ of rank one Higgs quotients inside $\Gr_2(E)$ is
\begin{multline}\label{eq:IA1}
I=\bigl(a_1w_1^2,\ a_2w_1^2,\ b_1w_1^2,\ b_2w_1^2,\ w_1\ell_1,\ w_1\ell_2\bigr) =  \\
 w_1\cdot\bigl(a_hw_1,\ b_hw_1,\ \ell_1,\ \ell_2\bigr)_{h=1,2},\quad \text{with}\quad \ell_h:=b_hw_2-a_hw_3 .
\end{multline}
\end{prop}
\begin{proof}
From equation \eqref{eq:matrices}, $\phi ^tw=(0,\ a\,w_1,\ b\,w_1)$, and the three equations $w_i(\phi ^tw)_j-w_j(\phi ^tw)_i=0$ of the transpose of equations  \eqref{eq:hgrsub} are
$$
(1,2):\ a\,w_1^2=0,\qquad (1,3):\ b\,w_1^2=0,\qquad (2,3):\ w_1\,(b\,w_2-a\,w_3)=0 .
$$
Splitting each $1$-form valued equation into its $dx_1$ and $dx_2$ components gives \eqref{eq:IA1}.
\end{proof}

Globally, $\HGr_2(\mathfrak  E)$ is the $\Ps^1$-bundle $\Ps(F^\vee)\simeq\Ps(F)$ of rank one quotients of $F=E/L$, with an embedded component supported on a section over the degeneracy divisor $D$, and with $\Ps^2$ fibres over $Z$. In particular, $\HGr_2(\mathfrak  E)_{\rm red}$ has a component isomorphic to $\Ps(F)$.

\parag{Case (1,2): rank one Higgs quotients}\label{sec:B1}

\begin{prop}\label{prop:B1}
Let $\mathfrak  E$ be of type   \textup{(1,2)}. The ideal of the Higgs Grassmannian $\HGr_2(\mathfrak  E)\subset\Ps(E^\vee)$ of rank one Higgs quotients is
\begin{equation}\label{eq:IB1}
I=\bigl(w_2\lambda_1,\ w_2\lambda_2,\ w_3\lambda_1,\ w_3\lambda_2\bigr)=(w_2,w_3)\cdot(\lambda_1,\lambda_2),\qquad \lambda_h:=a_hw_2+b_hw_3 .
\end{equation}
\end{prop}
\begin{proof}
From  equation \eqref{eq:matrices}, $\phi ^tw=(a\,w_2+b\,w_3,\ 0,\ 0)$, so the transpose of equations \eqref{eq:hgrsub} reads
$$
(1,2):\ -w_2\,(a\,w_2+b\,w_3)=0,\qquad (1,3):\ -w_3\,(a\,w_2+b\,w_3)=0,\qquad (2,3):\ 0=0 .
$$
Splitting into components one gets \eqref{eq:IB1}.
\end{proof}

Globally, over $X\setminus D$ the scheme $\HGr_2(\mathfrak  E)$ is the first infinitesimal neighbourhood of the section $\sigma_L\colon X\to\Ps(E^\vee)$ corresponding to the quotient $E\to L$, so that $\HGr_2(\mathfrak  E)_{\rm red}\simeq X$ and the fibres are zero-dimensional. Along $D$ the fibres jump to $\Ps^1$'s (the quotients of $E/\Im\phi $), with an embedded point at the section, and over $Z$ to $\Ps^2$'s. This Grassmannian has no dominant divisorial component.

\parag{Case (1,2): rank two Higgs quotients}\label{sec:B2}
Rank two quotients $E\to M$ correspond to their kernels, which are rank one subbundles; the relevant scheme is therefore $\HGr_1(\mathfrak  E)\subset\Ps E$, given by equation \eqref{eq:hgrsub} in the coordinates $(z_1,z_2,z_3)$, $v=\sum z_ke_k$.

\begin{prop}\label{prop:B2}
Let $\mathfrak  E$ be of type \textup{(1,2)}. The ideal of $\HGr_1(\mathfrak  E)\subset\Ps E$, the Higgs Grassmannian of rank two Higgs quotients, is
\begin{multline}\label{eq:IB2}
I=\bigl(a_1z_1^2,\ a_2z_1^2,\ b_1z_1^2,\ b_2z_1^2,\ z_1\ell_1,\ z_1\ell_2\bigr)= \\ z_1\cdot\bigl(a_hz_1,\ b_hz_1,\ \ell_1,\ \ell_2\bigr)_{h=1,2},\quad  \text{where}\quad \ell_h:=b_hz_2-a_hz_3 .
\end{multline}
\end{prop}
\begin{proof}
From \eqref{eq:matrices}, $\phi  v=(0,\ a\,z_1,\ b\,z_1)$, so the equations $z_i(\phi  v)_j-z_j(\phi  v)_i=0$ are
$$
(1,2):\ a\,z_1^2=0,\qquad (1,3):\ b\,z_1^2=0,\qquad (2,3):\ z_1\,(b\,z_2-a\,z_3)=0 ,
$$
and one splits into components.
\end{proof}

The ideal \eqref{eq:IB2} is actually the ideal \eqref{eq:IA1} of the rank one quotients of case (2,1), with $z$ in place of $w$. 

Globally, over $X\setminus D$ the scheme $\HGr_1(\mathfrak  E)$ is the $\Ps^1$-bundle $\Ps(F)$; along $D$ it acquires an embedded section (the line bundle $\Im\phi \subset F$); over $Z$ the fibres are $\Ps^2$. The rank two Higgs quotients are therefore $E\to E/K$ with $K\subset F$ a line subbundle; here $(F,0)$ is a Higgs subbundle of $\mathfrak  E$, and $E/K$ inherits the Higgs field $\bar e_1\mapsto[a\otimes e_2+b\otimes e_3]\in F/K\otimes\Omega^1_X$.

The fourth case, the rank two Higgs quotients of a system of type (2,1), is obtained in the same way: since $\phi_{(2,1)}=\phi_{(1,2)}^t$, equations \eqref{eq:hgrsub} for $\phi_{(2,1)}$ are the equations of Proposition~\ref{prop:B1} with $z$ in place of $w$, so the ideal of $\HGr_1(\mathfrak E)\subset\Ps E$ is $(z_2,z_3)\cdot(a_hz_2+b_hz_3)_{h=1,2}$. Hence over $X\setminus D$ the scheme $\HGr_1(\mathfrak E)$ is the first infinitesimal neighbourhood of the section $\Ps(L)\subset\Ps E$ given by the Higgs subbundle $(L,0)$, with fibres jumping to $\Ps^1$ along $D$ and to $\Ps^2$ over $Z$; its only component dominating $X$ is a section, of codimension two in $\Ps E$. This case provides no dominant divisorial component.

\begin{remark}\label{rem:dimsup}
\begin{enumerate}
% \item In case (1,2) the asymmetry between rank one and rank two quotients is striking: rank one Higgs quotients are scarce (generically the single quotient $E\to L$, carrying a nonreduced structure of length three), whereas rank two Higgs quotients form a whole, generically reduced, $\Ps^1$-bundle. This is the quotient-side counterpart of the dichotomy between (i) and (ii) in \cite[Prop.\ 6.1]{BGG26}, and case (2,1) inverts the picture.
\item The nonreduced structures  are deformation that are Higgs at first order, because the obstruction $v\wedge\phi (v)$ (or $w\wedge\phi ^tw$) is quadratic in the deformation parameter. 

\item Nothing in the computations uses that $X$ is a surface except that the number of $1$-form components of $\phi$  is two; for $\dim X=n$ and $F$ of rank two one has $2n$ generators of the same shape, and the role of $a\wedge b\neq0$ is played by the condition that the vectors $(a_h,b_h)_{h=1,\dots,n}$ span $\C^2$ at every point $x\in X$.
\end{enumerate}
\end{remark}

\section{The conjecture for systems of type (1,2) and (2,1)}
\label{sec:conjecture}\label{part:conj}

In this section $X$ is a smooth projective variety of dimension $n\ge 2$, and $\mathfrak  E=(L\oplus F,\phi )$ is a system of type $(1,2)$ or $(2,1)$, with $\phi \neq0$. We use the divisorial components of the Higgs Grassmannians of quotients  to settle, for these systems, the conjecture on curve semistable Higgs bundles.

Section~\ref{sec:A1} exhibits, in each of the two cases, a $\Ps^1$-bundle of Higgs quotients:
\begin{equation}\label{eq:Sigma}
\begin{aligned}
& \text{(2,1)}\quad s=1,\quad \Sigma=\{w_1=0\}=\Ps(F^\vee)\subset\Ps(E^\vee)=\Gr_2(E);
\\
& \text{(1,2)}\quad s=2,\quad \Sigma=\{z_1=0\}=\Ps(F)\subset\Ps E=\Gr_1(E).
\end{aligned}
\end{equation}
In case (2,1) the points of $\Sigma$ are the rank one quotients of $E$ factoring through $E\to E/L=F$; in case (1,2) they are the rank two quotients $E\to E/K$ with $K\subset F$ a line subbundle.

\begin{lemma}\label{lem:Sigma}
In both cases:
\begin{enumerate}
\item $\Sigma$ is a $\Ps^1$-bundle over $X$, hence integral and smooth of dimension $n+1$, it is a divisor in the $(n+2)$-dimensional variety $\Gr_{r-s}(E)$, and $\pi _s|_\Sigma\colon\Sigma\to X$ is surjective.
\item $\Sigma$ is an irreducible component of $\HGr_{r-s}(\mathfrak  E)$, and in fact $$\Sigma=\overline{\HGr_{r-s}(\mathfrak  E)_{\rm red}\cap\pi _s^{-1}(X\setminus Z)};$$ the other irreducible components of $\HGr_{r-s}(\mathfrak  E)$ lie over the vanishing locus $Z=\{\phi =0\}$ and do not dominate $X$.
\item Writing $\chi=c_1(Q_s)$, the class of $\Sigma \in A^1(\Gr_{r-s}(E), \Q)$ is
$$
[\Sigma]=\chi-\pi _s^*c_1(L)\ \ \text{in case (2,1)},\qquad
[\Sigma]=\chi-\pi _s^*c_1(F)\ \ \text{in case (1,2)} .
$$

\end{enumerate}
\end{lemma}

\begin{proof}
(i) is clear from equation \eqref{eq:Sigma}.

(ii) By Propositions \ref{prop:A1} and \ref{prop:B2} (and Remark~\ref{rem:dimsup} (ii) for $n>2$), over every point of $X\setminus Z$ the fibre of $\HGr_{r-s}(\mathfrak  E)_{\rm red}$ is exactly the fibre of $\Sigma$, namely the line $\{w_1=0\}$, resp.\ $\{z_1=0\}$, in the fibre $\Ps^2$ of a Grassmann bundle of $E$. Hence $\HGr_{r-s}(\mathfrak  E)_{\rm red}\cap\pi _s^{-1}(X\setminus Z)=\Sigma\cap\pi _s^{-1}(X\setminus Z)$. Since $\phi \ne0$, $X\setminus Z$ is a dense open subset of $X$, hence the claim.

(iii) In case (1,2), $\Sigma\subset\Ps E$ is the zero locus of the composition $\Oc_{\Ps E}(-1)\to\pi _s^*E\to\pi _s^*L$, i.e.\ of a section of $\Oc_{\Ps E}(1)\otimes\pi _s^*L$, so $[\Sigma]=h+\pi _s^*c_1(L)$ with $h=c_1(\Oc_{\Ps E}(1))$. The universal rank two quotient is $Q_2=\pi _s^*E/\Oc_{\Ps E}(-1)$, whence $\chi=\pi _s^*c_1(E)+h$ and $[\Sigma]=\chi-\pi _s^*c_1(F)$. In case (2,1), $\Sigma\subset\Ps(E^\vee)$ is the zero locus of $\Oc_{\Ps(E^\vee)}(-1)\to\pi _s^*E^\vee\to\pi _s^*L^{-1}$, a section of $\Oc_{\Ps(E^\vee)}(1)\otimes\pi _s^*L^{-1}$, so $[\Sigma]=h-\pi _s^*c_1(L)$; here the universal rank one quotient is $Q_1=\Oc_{\Ps(E^\vee)}(1)$, so $\chi=h$.
\end{proof}

\begin{prop}\label{prop:conj12}
Let $\mathfrak  E=(L\oplus F,\phi )$ be a system of type  \textup{(2,1)} or \textup{(1,2)} on a smooth projective variety $X$ of dimension $n$, and let $s$ be as in \eqref{eq:Sigma} ($s=1$ in case \textup{(2,1)}, $s=2$ in case \textup{(1,2)}). The following are equivalent:
\begin{enumerate}\itemsep=-2pt
\item[(i)] $\mathfrak  E$ is curve semistable;
\item[(ii)] $\mathfrak  E$ is semistable with respect to some polarization and the class $\theta_s(\mathfrak  E)$ restricted to $\HGr_{r-s}(\mathfrak  E)$ is nef;
\item[(iii)] $\mathfrak  E$ is semistable with respect to some polarization and $\Delta(E)=0$.
\end{enumerate}
In particular the Conjecture holds for all systems of  type $(1,2)$ and $(2,1)$, on smooth projective varieties of any dimension.
\end{prop}

\begin{proof}
Lemma \ref{lem:Sigma} shows that $\Sigma$ satisfies the hypotheses of Proposition \ref{prop:BGH}, which gives the claim.
\end{proof}

\section{Rank three --- the nilpotent case}
\label{sec:rank3nilp}

In this section we show that the conjecture on curve semistable Higgs bundles holds for all rank three Higgs bundles whose Higgs field is nilpotent, on smooth projective varieties of any dimension (actually, the same is true for a rank three Higgs bundle $\E=(E, \phi)$ such that $\phi -\frac13 \tr \phi \otimes \mathrm{id}_E$ is nilpotent). We shall see that the problem either reduces to the case of a system of type $(1,2)$ or $(2,1)$, as studied in Section~\ref{sec:conjecture}; or the Higgs field generically has a single Jordan block of size three. The second case is handled by means of a result of Bogomolov on line subbundles of $\Omega^1_X$.

Throughout, $X$ is a smooth projective variety of dimension $n$, and $\mathfrak  E=(E,\phi )$ is a Higgs bundle of rank $r=3$ on $X$. We say that $\phi $ is \emph{nilpotent} if for every local vector field $t$ the endomorphism $\phi _t=\langle\phi ,t\rangle$ of $E$ is nilpotent; equivalently, the coefficients $\sigma_i\in H^0(X,\textrm{Sym}^i\Omega^1_X)$ of the characteristic polynomial of $\phi $ vanish.

There are two situations, according to whether the Higgs field is nilpotent of order two or three.
For $k\ge1$ let
$\phi^{(k)}\colon E\to E\otimes\operatorname{Sym}^k \Omega^1_X$
be the $k$-fold iterate of $\phi$.
Moreover $\phi_t$ will denote the contraction of $\phi$ with a local vector field $t$.  We set
\begin{equation}\label{eq:KI}
K=\ker\phi ,\qquad I=\sum_t\operatorname{Im} \phi_t ,
\end{equation}
the sum being over all local vector fields $t$; $I$ is the image of the contraction $E\otimes T_X\to E$ defined by $\phi$. The subsheaf $K$ is saturated, as $E/K$ embeds into the locally free sheaf $E\otimes\Omega^1_X$; the subsheaf $I$ need not be, and we write $\bar I$ for its saturation in $E$.

\begin{lemma}\label{lem:dichotomy}
Let $\phi$ be nilpotent and nonzero. Then $\phi^{(3)}=0$, and   one of the following holds:
\begin{enumerate}
\item[\textup{(N1)}] $\phi^{(2)}=0$, i.e.\ $\phi_s\phi_t=0$ for all local vector fields $s,t$. In this case $I\subset K$ and $1\le\rk I\le\rk K\le2$.
\item[\textup{(N2)}] $\phi^{(2)}\ne0$. In this case some $\phi_t$ is a regular nilpotent endomorphism at the generic point of $X$, and
\begin{equation}\label{eq:flag}
0\subset E_1:=\ker\phi\subset E_2:=\ker\phi^{(2)}=\phi^{-1}(E_1\otimes\Omega^1_X)\subset E
\end{equation}
is a flag of saturated subsheaves with $\rk E_i=i$ and $\phi(E_i)\subset E_{i-1}\otimes\Omega^1_X$ for $i=1,2,3$, where $E_0=0$ and $E_3=E$.
\end{enumerate}
\end{lemma}

\begin{proof}
At a point $x\in X$ the endomorphisms $\phi_t(x)$ of $E_x$, $t\in T_xX$, are nilpotent and commute with one another, the latter being the commutativity condition $\phi\wedge\phi=0$. A commuting family of nilpotent endomorphisms of a vector space has a common kernel vector, hence, inductively, a common invariant flag on whose graded pieces $\phi$  acts by zero; on the three-dimensional space $E_x$ the $\phi_t(x)$ are therefore simultaneously strictly triangular, and every product of three of them vanishes. Thus $\phi^{(3)}=0$. 

\textup{(N1)}   As $\phi\ne0$ we have $I\ne0$ and $K\ne E$, whence the bounds on the ranks.

\textup{(N2)} Choose $s,t$ with $\phi_s\phi_t\ne0$ at the generic point. If $\phi_s^2$ and $\phi_t^2$ both vanish there, then $(\phi_s+\phi_t)^2=2\phi_s\phi_t\ne0$; so in any case there is a $u$ with $\phi_u^2\ne0$ at the generic point. A nilpotent endomorphism of a three-dimensional space with nonzero square is regular, so at the generic point $\ker\phi_u$, $\ker\phi_u^2$ and $\Im\phi_u$ have dimensions $1$, $2$ and $2$. Now $E_1\subset\ker\phi_u$, and $E_1\supset\Im(\phi_s\phi_t)\ne0$ because $\phi^{(3)}=0$; hence $\rk E_1=1$. Likewise $E_2\subset\ker\phi_u^2$, and $E_2\supset\Im\phi_u$ because $\phi^{(3)}=0$; hence $\rk E_2=2$. Both $E_1$ and $E_2$ are kernels of morphisms to locally free sheaves, so they are saturated. The equality $\ker\phi^{(2)}=\phi^{-1}(E_1\otimes\Omega^1_X)$ is again $\ker(\phi\otimes1)=E_1\otimes\Omega^1_X$, and it gives $\phi(E_2)\subset E_1\otimes\Omega^1_X$; in the same way $\phi^{(3)}=0$ gives $\phi(E)\subset E_2\otimes\Omega^1_X$, and $\phi(E_1)=0$ by definition.
\end{proof}

 Since $K$ and the $E_i$ are saturated in a locally free sheaf, they are reflexive \cite[Lemma 1.1.16]{OSS11}, hence locally free when $n=2$; but they need not be subbundles. In case \textup{(N2)} the quotients $E_2/E_1$ and $E/E_2$ are torsion free of rank one and, on a surface, of the form $L_2\otimes I_{Z_2}$ and $L_3\otimes I_{Z_3}$ with $L_i$ line bundles and $Z_i$ zero-dimensional subschemes. In case \textup{(N1)} all three pairs $(\rk I,\rk K)=(1,1),(1,2),(2,2)$ occur; Proposition \ref{lem:N1} below treats the cases $\rk K=2$ and $\rk I=1$, which exhaust the possibilities, and when $\rk I=1$, $\rk K=2$ both of its descriptions apply.

\parag{Case \textup{(N1)}: systems of type $(1,2)$ and $(2,1)$ up to extension}

\begin{prop}\label{lem:N1}
Let $\mathfrak E$ be as in case \textup{(N1)}, with $K$, $I$ as in \eqref{eq:KI} and $\bar I$ the saturation of $I$. Then $\phi$ factors as
$$
E\to E/K\xrightarrow{\ \psi\ }I\otimes\Omega^1_X\hookrightarrow K\otimes\Omega^1_X\hookrightarrow E\otimes\Omega^1_X ,
$$
and:
\begin{enumerate}
\item[(i)] if $\rk K=2$, let $U\subset X$ be the dense open set where $\phi\ne0$ and $K$ is a subbundle of $E$, and let $\Sigma\subset\Ps E$ be the closure of $\Ps(K|_U)$. Then $\Sigma$ is an irreducible component of the Higgs Grassmannian $\HGr_1(\mathfrak E)\subset\Ps E$ of rank two Higgs quotients, it is a divisor in $\Ps E$ and it surjects onto $X$. If $K$ is a subbundle of $E$, then $\Sigma=\Ps(K)$;
\item[(ii)] if $\rk I=1$, then $\bar I$ is a line bundle. Let $U\subset X$ be the dense open set where $\phi\ne0$ and $I$ is a subbundle of $E$, so that $I|_U=\bar I|_U$, and let $\Sigma\subset\Ps(E^\vee)$ be the closure of $\Ps((E/\bar I)^\vee|_U)$. Then $\Sigma$ is an irreducible component of the Higgs Grassmannian $\HGr_2(\mathfrak E)\subset\Ps(E^\vee)$ of rank one Higgs quotients, it is a divisor in $\Ps(E^\vee)$ and surjects onto $X$. If $\bar I$ is a subbundle of $E$, then $\Sigma=\Ps((E/\bar I)^\vee)$.
\end{enumerate}
Moreover, $\mathfrak E$ is curve semistable if and only if it is semistable with respect to some polarization and $\Delta(E)=0$.
\end{prop}

\begin{proof}
The factorization follows from $\phi(E)\subset I\otimes\Omega^1_X\subset K\otimes\Omega^1_X$ and $\phi(K)=0$. Since $K$ and $\bar I$ are saturated, $E/K$ and $E/\bar I$ are torsion free, hence locally free outside a closed subset of codimension $\ge2$; so $K$ and $\bar I$ are subbundles of $E$ on a dense open set, and so is $I$, which agrees with $\bar I$ on a dense open set. As $\phi\ne0$, the sets $U$ of (i) and (ii) are dense and open. In case (ii), $\bar I$ is a saturated rank one subsheaf of a locally free sheaf, hence reflexive, hence invertible since $X$ is smooth \cite[Prop.~1.9]{Hart-refl}.

(i) Let $x\in U$. Lines in $K_x$ are invariant, since $\phi(K)=0$. If $v\notin K_x$, invariance of $\langle v\rangle$ would give $\phi_t(v)\in\langle v\rangle\cap K_x=0$ for all $t$, i.e.\ $\psi_x(\bar v)=0$ with $\bar v$ spanning $(E/K)_x$; thus $\psi_x=0$ and $\phi_x=0$, contrary to $x\in U$. So over $U$ the reduced scheme $\HGr_1(\mathfrak E)_{\rm red}$ is the $\Ps^1$-bundle $\Ps(K|_U)$, and its closure $\Sigma$ is an irreducible component of $\HGr_1(\mathfrak E)_{\rm red}$ of dimension $n+1$, hence a divisor in $\Ps E$ dominating $X$. If $K$ is a subbundle, then $\Sigma=\Ps(K)$.

(ii) Let $x\in U$, so that $I_x=\bar I_x\subset E_x$ is the line spanned by the images of the endomorphisms $\phi_t(x)$, $t\in T_xX$. A plane $W\subset E_x$ containing $I_x$ is invariant because $\phi(E)\subset I\otimes\Omega^1_X$. Conversely, if $W\not\supset I_x$ then $W\cap I_x=0$, so invariance $\phi_x(W)\subset W\otimes\Omega^1_{X,x}$ together with $\phi_x(W)\subset I_x\otimes\Omega^1_{X,x}$ gives $\phi_x(W)=0$, i.e.\ $W\subset\ker\phi_x$. Since $\phi_x\ne0$, $\ker\phi_x$ is a proper subspace of $E_x$, so $W=\ker\phi_x$; but $\ker\phi_x=\bigcap_t\ker\phi_t(x)$ contains $\sum_t\Im\phi_t(x)=I_x$, because $\phi_s(x)\phi_t(x)=0$ for all $s,t$, contradicting $W\not\supset I_x$. So over $U$ the fibres of $\HGr_2(\mathfrak E)_{\rm red}$ are exactly the planes through $I_x$, i.e.\ those of the $\Ps^1$-bundle $\Ps((E/\bar I)^\vee|_U)$, and one concludes as in (i).

By Lemma~\ref{lem:dichotomy} at least one of \textup{(i)}, \textup{(ii)} holds true (both when $\rk I=1$, $\rk K=2$), so in every case some Higgs Grassmannian of $\mathfrak E$ has a dominating divisorial component, and Proposition~\ref{prop:BGH} applies, with $s=2$ in case (i) and $s=1$ in case (ii).
\end{proof}

% \begin{remark}
% In Huang's counterexample (see Introduction), here $E$ is a non-split extension $0\to F^*\to E\to F\to0$ on the symmetric square of a plane quintic, with $\phi $ the composite $E\to F\xrightarrow{\psi}F^*\otimes\Omega^1_X\to E\otimes\Omega^1_X$, so $\phi ^{(2)}=0$, and both $K=\ker\phi$ and the saturation $\bar I$ of the image of $\phi$ are equal to $F^*$, of rank two. The mechanism of Proposition \ref{lem:N1} requires $\rk\bar I=1$ (planes through the image line) or $\rk K=r-1$ (lines in the kernel); for Huang's bundle $\rk\bar I=\rk K=2=r-2$ and neither holds, so that the sections $\Ps(F^*)\subset\Ps E$ and $\Ps(F^\vee)\subset\Ps(E^\vee)$ have codimension two and no divisorial component is available, as Section~\ref{sec:hgr} says it cannot be. Since $\rk\bar I\le\rk K\le r-1$, the equality $\bar I=K$ with $\rk K\notin\{1,r-1\}$ needs $r\ge4$; in rank three every Higgs bundle with $\phi ^{(2)}=0$ has a dominating divisorial component and the Conjecture holds.
% \end{remark}

\parag{Case \textup{(N2)} on a surface}

In this paragraph $X$ is a smooth projective surface and $\mathfrak  E$ is in case (N2), with canonical flag \eqref{eq:flag}. We write
\begin{multline*}
E_2/E_1=L_2\otimes I_{Z_2},\qquad E/E_2=L_3\otimes I_{Z_3},\\ L_1:=E_1,\qquad \lambda_i=c_1(L_i),\qquad \zeta_i=\operatorname{length}(Z_i)\ge0,
\end{multline*}where the $L_i$ are line bundles. Then \[c_1(E)=\lambda_1+\lambda_2+\lambda_3 \quad \text{and} \quad c_2(E)=\lambda_1\lambda_2+\lambda_1\lambda_3+\lambda_2\lambda_3+\zeta_2+\zeta_3,\] whence
\begin{equation}\label{eq:Delta3}
\Delta(E)=-\frac16\Big[(\lambda_1-\lambda_2)^2+(\lambda_1-\lambda_3)^2+(\lambda_2-\lambda_3)^2\Big]+\zeta_2+\zeta_3 .
\end{equation}
Note that $\zeta_2=0$ if and only if $E_1$ is a subbundle of $E_2$, and $\zeta_3=0$ if and only if $E_2$ is a subbundle of $E$. If $E_1$ is a subbundle of $E$, i.e.\ if $E/E_1$ is locally free, then $\zeta_2=0$, because $E_2/E_1$ is a saturated rank one subsheaf of the rank two bundle $E/E_1$, hence reflexive, hence a line bundle. So the flag $E_1\subset E_2\subset E$ consists of subbundles if and only if $\zeta_2=\zeta_3=0$, and then in particular $E_1$ is a subbundle of $E$. We also introduce the classes
\begin{equation}\label{eq:uv}
u:=c_1(E)-3\,c_1(E_1)=\lambda_2+\lambda_3-2\lambda_1,\qquad v:=2\,c_1(E)-3\,c_1(E_2)=2\lambda_3-\lambda_1-\lambda_2 .
\end{equation}
% (When $c_1(E)=0$ and the flag splits, these are, up to the factor $3$, the nef classes $b$ and $a$ of Section~\ref{subsec:rk3}: $u=3b$, $v=3a$.) 
Then
\begin{equation}\label{eq:Deltauv}
\Delta(E)=-\frac19\,(u^2-uv+v^2)+\zeta_2+\zeta_3 .
\end{equation}

\begin{lemma}\label{lem:nef}
If $\mathfrak  E$ is curve semistable, the classes $u$ and $v$ are nef.
\end{lemma}

\begin{proof}
Let $f\colon C\to X$ be a morphism from a smooth projective curve, and let $F_i\subset f^*E$ be the saturation of the image of $f^*E_i\to f^*E$. This map is generically injective, since $E_i$ is a subbundle of $E$ outside a finite set, and its kernel is a torsion subsheaf of the locally free sheaf $f^*E_i$ on $C$, hence zero; so $F_i$ is a subbundle of rank $i$ with $\deg F_i\ge\deg f^*E_i$. Since $\phi (E_i)\subset E_{i-1}\otimes\Omega^1_X$, the image of $f^*E_i$ is invariant under $f^*\phi $ (the Higgs field of $f^*\mathfrak  E$ is the composition of $f^*\phi $ with $f^*\Omega^1_X\to\Omega^1_C$), and so is its saturation $F_i$. Semistability of $f^*\mathfrak  E$ gives $\deg f^*E_1\le\deg F_1\le\frac13\deg f^*E$ and $\deg f^*E_2\le\deg F_2\le\frac23\deg f^*E$, i.e.\ $u\cdot f_*C\ge0$ and $v\cdot f_*C\ge0$.
\end{proof}

Now we analyze the structure of the graded pieces of the Higgs field. As $\phi(E_i)\subset E_{i-1}\otimes\Omega^1_X$, the Higgs field induces nonzero sheaf morphisms  
$$
\psi_1\colon L_2\otimes I_{Z_2}\to L_1\otimes\Omega^1_X,\qquad
\psi_2\colon L_3\otimes I_{Z_3}\to L_2\otimes I_{Z_2}\otimes\Omega^1_X .
$$
  Since $\mathrm{Hom}(L\otimes I_Z,F)\simeq \mathrm{Hom}(L,F)$ for a locally free sheaf $F$ and a zero-dimensional subscheme $Z$, we obtain global sections 
$$
\psi_1\in H^0(X,L_1L_2^{-1}\otimes\Omega^1_X),\qquad
\psi_2\in H^0(X,L_2L_3^{-1}\otimes\Omega^1_X),
$$
which we still denote by the same symbols; they are nonzero because so they are on the dense open subset 
$
U=X\setminus(Z_2\cup Z_3)
$.

\begin{lemma}\label{lem:foliation}
There exist a saturated line subbundle $N\subset\Omega^1_X$, and nonzero sections $\alpha\in H^0(L_1L_2^{-1}N)$, $\beta\in H^0(L_2L_3^{-1}N)$ such that $\psi_1=\omega\otimes\alpha$ and $\psi_2=\omega\otimes\beta$, where $\omega\in H^0(\Omega^1_X\otimes N^{-1})$ is the inclusion. Let    $\nu=c_1(N)$. Then,  setting $\alpha'=\lambda_1-\lambda_2+n$ and $\beta'=\lambda_2-\lambda_3+n$, which are classes of effective divisors, one has
\begin{equation}\label{eq:3n}
3\,\nu=u+2\alpha'+\beta'=v+\alpha'+2\beta' .
\end{equation}
\end{lemma}

\begin{proof}
We first show that $\psi_1\wedge\psi_2=0$ as a section of $L_1L_3^{-1}\otimes\Omega^2_X$ on $X$. Since this sheaf is torsion free and $U$ is dense, it is enough to check it on $U$. We may assume $U$ is affine; over it the bundle $E$ admits a local splitting $E=L_1\oplus L_2\oplus L_3$ compatible with the flag, and with respect to such a splitting, with local frame $e_1,e_2,e_3$, one has $\phi(e_1)=0$,
$\phi(e_2)=\psi_1\otimes e_1$ and $\phi(e_3)=\psi_2\otimes e_2+\omega'\otimes e_1$ for some
local $1$-form $\omega'$. Hence $\phi\wedge\phi(e_3)=\psi_2\wedge\psi_1\otimes e_1$, all other
components of $\phi\wedge\phi$ vanishing for degree reasons, and the condition
$\phi\wedge\phi=0$ gives $\psi_1\wedge\psi_2=0$ on $U$, hence on $X$.

 The vanishing of $\psi_1\wedge\psi_2$ says that the morphism of locally free sheaves
$$
\Psi=(\psi_1,\psi_2)\colon L_2L_1^{-1}\oplus L_3L_2^{-1}\longrightarrow\Omega^1_X
$$
has rank $\le1$ at every point of $X$, and it has rank one at the generic point since $\psi_1\neq0$.
Its image is therefore a rank one subsheaf of $\Omega^1_X$; let $N$ be its saturation. A saturated subsheaf of a locally free sheaf is reflexive, and a reflexive sheaf of rank one on a smooth variety is a line bundle, so $N$ is a line subbundle of $\Omega^1_X$, say with inclusion $\omega\colon N\hookrightarrow\Omega^1_X$. Since $\Psi$ takes values in $N$, its components factor as $\psi_1=\omega\circ\alpha$ and
$\psi_2=\omega\circ\beta$ for unique morphisms  $\alpha\in H^0(L_1L_2^{-1}N)$ and
$\beta\in H^0(L_2L_3^{-1}N)$; both are nonzero because $\psi_1,\psi_2$ are. This is the first assertion. The classes $\alpha'=\lambda_1-\lambda_2+\nu$ and $\beta'=\lambda_2-\lambda_3+\nu$ are those of the   divisors of zeroes of the nonzero sections $\alpha$ and $\beta$, hence effective, and the formulas \eqref{eq:3n} follow from \eqref{eq:uv}:
$u=-2(\lambda_1-\lambda_2)-(\lambda_2-\lambda_3)=-2(\alpha'-\nu)-(\beta'-\nu)$, and similarly for $v$.
\end{proof}

We now use a theorem of Bogomolov \cite{Bogomolov78} and Sommese \cite{SS}\footnote{This result also appears in \cite[Cor.~6.9]{EV92}, in Reid's exposition  of Bogomolov's inequality \cite[Theorem 2]{Reid77}, in \cite[Theorem 8.3]{Greb-et-al},
and, exactly as stated below, in \cite[Theorem 1.1]{Graf}.
Actually the result as stated here does not appear explicitly in \cite{SS} but it is rather a consequence of Theorems 3.36 and 3.37 there, as spelled out in \cite{EV92}.} that an invertible subsheaf of $\Omega^p_X$ has Kodaira-Iitaka dimension at most $p$
(the Kodaira-Iitaka dimension  $\kappa=\kappa(X,N)$ of a line bundle $N$ on a variety $X$ is the unique $\kappa$ such that there exist constants $0<a<b$ for which
$$ a m^\kappa \le h^0(X,N^m) \le bm^\kappa$$
for $m\gg 0$, with the convention  that $\kappa(X,N)=-\infty$ if  $ h^0(X,N^m)=0$ for $m>0$). If $\kappa(X,N)=\dim X$ then $N$ is said to be \emph{big}; if $N$ is nef, that is equivalent to
$c_1(N)^{\dim X}>0$.

\begin{lemma}[Bogomolov's lemma]\label{lem:bogomolov}
Let $X$ be a smooth projective surface and $N\subset\Omega^1_X$ an invertible subsheaf. Then $\kappa(X,N)\le1$; in particular $N$ is not big.
\end{lemma}

\begin{prop}\label{prop:uv}
If $\mathfrak  E$ is curve semistable, then $u^2=uv=v^2=0$, and therefore $\Delta(E)=\zeta_2+\zeta_3$.
\end{prop}

\begin{proof}
By Lemma \ref{lem:nef} the classes $u,v$ are nef. If $u^2>0$, then $u$ is nef and big, and by \eqref{eq:3n} $N^{\otimes3}$ is the sum of a nef and big class and an effective divisor, hence big \cite[\S 2.2]{Laza2004}, contradicting Lemma \ref{lem:bogomolov}. So $u^2=0$, and likewise $v^2=0$. If $uv>0$, then $u+v$ is nef with $(u+v)^2=2uv>0$, hence big, and $N^{\otimes6}=(u+v)+3\alpha'+3\beta'$ is big, again a contradiction. Hence $u^2=uv=v^2=0$, and \eqref{eq:Deltauv} gives $\Delta(E)=\zeta_2+\zeta_3$.
\end{proof}

It remains to show that the flag \eqref{eq:flag} is a flag of subbundles. Now the components of the Higgs Grassmannians dominating $X$ are only sections of codimension two (see \cite[Thm.~4.5]{BGG26}), so  Proposition~\ref{prop:BGH} does not apply, but the nefness of $\theta_s(\mathfrak  E)$ on these sections is enough.

\begin{prop}\label{prop:subbundles}
If $\mathfrak  E$ is curve semistable, then the canonical flag \eqref{eq:flag} is a flag of subbundles: $E_2$ is a subbundle of $E$, and $E_1$ is a subbundle of  $E_2$. In particular $\zeta_2=\zeta_3=0$.
\end{prop}

\begin{proof}
Consider the rank two Higgs quotients, parameterized by $\HGr_1(\mathfrak  E)\subset\Ps E$, with $\pi \colon\Ps E\to X$ and universal rank two quotient $Q_2=\pi ^*E/\Oc_{\Ps E}(-1)$. Over the open set $U\subset X$ where $E_1$ is a subbundle, the inclusion $E_1\subset E$ defines a section of $\Ps E|_U$ whose image lies in $\HGr_1(\mathfrak  E)$ (lines in $E_1$ are $\phi $-invariant since $\phi (E_1)=0)$. Let $\Sigma_1\subset\Ps E$ be the closure of this image, an irreducible surface contained in $\HGr_1(\mathfrak  E)$, and let $\tilde\rho\colon\widetilde\Sigma_1\xrightarrow{j}\Sigma_1\to X$ be a resolution of singularities of $\Sigma_1$ composed with the projection; $\tilde\rho$ is a birational morphism of smooth projective surfaces, an isomorphism over $U$. We claim that the  pullback of the tautological inclusion gives an injection $\tilde\rho^*E_1\to\tilde\rho^*E$ whose saturation is the line subbundle $M=j^\ast\Oc_{\Ps E}(-1)|_{\Sigma_1}$, and $M=
\tilde\rho^*E_1(e)$ for an effective divisor $e$ supported in the exceptional locus of $j$. To show this, recall
that the tautological inclusion $\Oc_{\Ps E}(-1)\subset\pi^*E$ has, over a point
$\ell\in\Ps(E_x)$, fibre the line $\ell\subset E_x$ itself.
The map $\tilde\rho^*E_1\to\tilde\rho^*E$ is a
nonzero map from a line bundle to a locally free sheaf, hence injective, but at a
point $y$ over $x\in X\setminus U$ the induced map on fibres $E_1(x)\to E(x)$ is
zero, so its image is not saturated along $\tilde\rho^{-1}(X\setminus U)$.
On the other hand $M=j^*\Oc_{\Ps E}(-1)|_{\Sigma_1}$ is a line subbundle of
$\tilde\rho^*E$, being the pullback of a subbundle, with fibre
$M_y=\ell_y$ at the point $y=\ell_y\subset E_{\pi(y)}$. Over $U$ the
points of $\Sigma_1$ are the lines $E_{1,x}\subset E_x$, so $M$ and the image of
$\tilde\rho^*E_1$ agree there; since $\tilde\rho^*E/M$ is torsion-free, the composite
$\tilde\rho^*E_1\to\tilde\rho^*E/M$, which vanishes on a dense open set, is zero, so
$\tilde\rho^*E_1\subset M$ and $M$ is the saturation.
The inclusion $\tilde\rho^*E_1\hookrightarrow M$ of line bundles is a nonzero section of
$M\otimes\tilde\rho^*E_1^{-1}$; if    $e$ is  its divisor of zeroes, then
$M\cong\tilde\rho^*E_1(e)$. As the inclusion is an isomorphism over $U$,
$e$ is supported in $\tilde\rho^{-1}(X\setminus U)$, a union of curves lying over
finitely many points of $X$, contracted by $\tilde\rho$. Conversely, if $E_1$ is not a subbundle at $x$, the section of $M\otimes\tilde\rho^*E_1^{-1}$ vanishes at every point of $\tilde\rho^{-1}(x)$, so $e\ne0$; thus $e=0$ if and only if $E_1$ is a subbundle of $E$. By the projection formula
$e\cdot\tilde\rho^*\gamma=0$ for every divisor class $\gamma$ on $X$, and $e^2<0$ if $e\neq0$, as the intersection form on the exceptional curves of the birational morphism $\tilde\rho$ is negative definite \cite[III.2.1]{BHPV04}.

Now compute the pullback of $\theta_2(  E)=c_1(Q_2)-\tfrac23\pi ^*c_1(E)$ to $\widetilde\Sigma_1$. Since \[j^\ast c_1(Q_2)|_{\Sigma_1}=\tilde\rho^*c_1(E)-c_1(M)=\tilde\rho^*c_1(E)-\tilde\rho^*\lambda_1-e,\] we get
$$
j^\ast\theta_2(\mathfrak  E)|_{\Sigma_1}=\tfrac13\,\tilde\rho^*u-e ,
$$
and, as $e\cdot\tilde\rho^*(\gamma)=0$ for every class $\gamma$ on $X$,
\begin{equation}\label{eq:theta2sq}
\int_{\widetilde\Sigma_1}j^\ast\theta_2(\mathfrak  E)^2=\tfrac19\,u^2+e^2=e^2
\end{equation}
by Proposition \ref{prop:uv}. Since $\mathfrak  E$ is curve semistable, $\theta_2(\mathfrak  E)$ is nef on $\HGr_1(\mathfrak  E)$ \cite{BHR06,BGO11}, hence on $\Sigma_1$ and on $\widetilde\Sigma_1$, so the left hand side of \eqref{eq:theta2sq} is nonnegative and then $e=0$. So $E_1$ is a subbundle of $E$, i.e., $E/E_1$ is locally free, so that $E_2/E_1$, being saturated in it, is a line bundle and then $\zeta_2=0$.

The same argument applied to the rank one Higgs quotients, parameterized by $\HGr_2(\mathfrak  E)\subset\Ps(E^\vee)$, with the closure $\Sigma_2$ of the section defined by $E\to E/E_2$ over the locus where $E_2$ is a subbundle, gives $j^\ast Q_1|_{\Sigma_2}=\tilde\rho^*L_3(-e')$ with $e'$ supported in the exceptional divisor of $j\colon\tilde\Sigma_2\to\Sigma_2$. Hence \[j^\ast\theta_1(\mathfrak  E)|_{\Sigma_2}=\tilde\rho^*l_3-e'-\tfrac13\tilde\rho^*c_1(E)=\tfrac13\tilde\rho^*v-e'\] and \[\int_{\widetilde\Sigma_2}j^\ast\theta_1(\mathfrak  E)^2=\tfrac19v^2+e'^2=e'^2\ge0,\] so $e'=0$, and $E_2$ is a subbundle of $E$, that is, $\zeta_3=0$. Together with the first part,we have that $E_1\subset E_2\subset E$ is a flag of subbundles.
\end{proof}

\begin{prop}\label{prop:N2surface}
Let $\mathfrak  E$ be a rank three Higgs bundle on a smooth projective surface in case \textup{(N2)}. If $\mathfrak  E$ is curve semistable then $\Delta(E)=0$.
\end{prop}

\begin{proof}
Combine Propositions \ref{prop:uv} and \ref{prop:subbundles} with Equation \eqref{eq:Deltauv}.
\end{proof}

\parag{Proof of the main result} Finally, we collect the previous statements to provide a proof of Theorem \ref{thm:main}.
We may assume $\phi\ne 0$ and $\dim X = n>1$.
The ``if'' implication holds for every Higgs bundle. For the  ``only if,'' recall that curve semistability implies semistability with respect to every polarization.  For $n=2$, in case (N1) the claim is Proposition \ref{lem:N1}, and in case (N2) it is Proposition \ref{prop:N2surface}. The result for $n>2$ is then established as in  Proposition \ref{prop:BGH}.
 
\begin{remark}
    In spite of Huang's counterexample, the techniques of this paper allow one to prove the Conjecture in the nilpotent rank four case for several types of the graded bundle of the nilpotent filtration.
\end{remark}

\let\oldthebibliography\thebibliography
\renewcommand{\thebibliography}[1]{%
  \oldthebibliography{#1}%
  \setlength{\itemsep}{3pt}%
  \setlength{\parskip}{0pt}%
}

 \frenchspacing

\providecommand{\bysame}{\leavevmode\hbox to3em{\hrulefill}\thinspace}
\providecommand{\MR}{\relax\ifhmode\unskip\space\fi MR }
\providecommand{\MRhref}[2]{%
  \href{http://www.ams.org/mathscinet-getitem?mr=#1}{#2}
}
\providecommand{\href}[2]{#2}

\end{document}